\documentclass[a4paper]{amsart}
\usepackage{amscd,amssymb,amsopn,amsmath,amsthm,graphics,amsfonts,enumerate,verbatim,calc}
\usepackage[dvips]{graphicx}
\usepackage[colorlinks=true,linkcolor=blue,citecolor=blue]{hyperref}
\usepackage{showlabels}
\input xy
\xyoption{all}

\headsep=1cm \numberwithin{equation}{section}
\newcommand{\To}{\rightarrow}

\newtheorem{theorem}{Theorem}[section]
\newtheorem{corollary}[theorem]{Corollary}
\newtheorem{lemma}[theorem]{Lemma}
\newtheorem{proposition}[theorem]{Proposition}

\theoremstyle{definition}
\newtheorem{definition}[theorem]{Definition}

\newtheorem{remark}[theorem]{Remark}

\theoremstyle{plain}

\theoremstyle{definition}

\numberwithin{equation}{section}

\begin{document}

\title{Gorenstein injectivity of the section functor}
\author{Reza Sazeedeh}

\address{Department of Mathematics, Urmia University, P.O.Box: 165, Urmia, Iran-And\\
School of Mathematics, Institute for Research in Fundamental Sciences (IPM),
P. O. Box: 19395-5746, Tehran, Iran} \email{rsazeedeh@ipm.ir}


\keywords{Strongly cotorsion, Gorenstein injective, Gorenstein flat,
Local cohomology.\\This research was in part supported by a grant from IPM (No. 87200025).}

\begin{abstract}

Let $R$ be a commutative Noetherian ring of Krull
dimension $d$ admitting a dualizing complex $D$ and let $\frak a$ be
any ideal of $R$, we prove that $\Gamma_{\frak a}(G)$ is Gorenstein
injective for any Gorenstein injective $R$-module $G$.
Let $(R,\frak m)$ be a local ring and $M$ be a finitely generated $R$-module.
 We show that ${\rm Gid}{\bf R}\Gamma_{\frak
m}(M)<\infty$ if and only if ${\rm Gid}_{\hat{R}}(M\otimes_R\hat{R})<\infty$. We also show that if ${\rm Gfd}_R{\bf
R}\Gamma_{\frak m}(M)<\infty$, then ${\rm Gfd}_RM<\infty$. Let $(R,\frak m)$ be a Cohen-Macaulay local ring and $M$ be a Cohen-Macaulay module of dimension $n$. We prove that if $H_{\frak m}^n(M)$ is of finite G-injective dimension, then Gid$_RH_{\frak m}^n(M)=d-n$. Moreover, we
prove that if $M$ is a Matlis reflexive strongly torsion free module of finite G-flat dimension, then Gfd$_R\hat{M}<\infty$, where $\hat{M}$
is $\frak m$-adic completion.
\end{abstract}

\maketitle

\tableofcontents

\section{Introduction}
\hspace{0.4cm} Throughout this paper, $R$ is a commutative
Noetherian ring (with identity). Grothendieck local cohomology
theory is an effective tool for mathematicians working in the theory
of commutative algebra and in algebraic geometry. There are several
ways to compute these cohomology modules. For instance, following
[BS], they can be computed by use of the right derived functors of
the $\frak a$-torsion functor $\Gamma_{\frak
a}(-)=\bigcup_{n\in\mathbb{N}}(0:_{(-)}{\frak a}^n)$, where $\frak
a$ is an arbitrary ideal of $R$. One of the main theorem in local
cohomology subject is Grothendieck's local duality theorem. Let
$(R,\frak m)$ be a local ring of dimension $d$ admitting a dualizing
module $\omega$. The local duality theorem establishs a relation
between the functors $\Gamma_{\frak m}(-)$ and Hom$_R(-,\omega)$. So
one may study the homological dimensions of $R$-modules in terms of
the study of local cohomology.

\hspace{0.4cm} The author in [S3] studied the effect of the section
functor $\Gamma_{\frak a}(-)$ on the Auslander and Bass classes of
$R$-modules, specially when the module was Gorenstein injective,
Gorenstein flat or maximal Cohen-Macaulay and $\frak a$ was any
ideal of $R$. In this paper we will continue this argument for some
$R$-modules which are more general.

\hspace{0.4cm} In this paper we show that if $R$ admits a dualizing complex $D$ and $\frak a$ is any ideal of $R$
, then  $\Gamma_{\frak a}(G)$ is Gorenstein injective for any Gorenstein injective module $G$ (see Theorem 3.2).
  As a conclusion of this theorem we show that if $X$ is a homologically left bounded complex in $D(R)$ and
  of finite Gorenstein injective dimension, then ${\rm Gid}_R{\bf R}\Gamma_{\frak a}(X)\leq {\rm Gid}_R X$ (see Corollary 3.3). We also prove that $-{\rm inf}{\bf R}\Gamma_{\frak a}(X)\leq {\rm Gid}_R X$  for any homologically left bounded complex $X$ (see Theorem 3.4).

  Let $(R,\frak m)$ be a local ring. Another principal aim of this paper is to study of the Gorenstein homological
dimension of finitely generated $R$-modules in terms of the
Gorenstein homological dimension of its local cohomology modules at
the maximal ideal $\frak m$. We show that if $M$ is a finitely
generated $R$-module, then ${\rm Gid}_R{\bf R}\Gamma_{\frak
m}(M)<\infty$ if and only if ${\rm Gid}_{\hat{R}}(M\otimes_R\hat{R})<\infty$.
(see Theorem 3.3). In particular, if $M$ is a non-zero cyclic $R$-module such that ${\rm Gid}_R{\bf R}\Gamma_{\frak
m}(M)<\infty$, then $R$ is Gorenstein (see
Corollary 3.8). We also show that if $M$ is a finitely
generated $R$-module such that ${\rm Gfd}_R{\bf R}\Gamma_{\frak
m}(M)<\infty$, then ${\rm Gfd}_{R}M<\infty$ (see Theorem 3.9).

\hspace{0.4cm} Assume that $(R,\frak m)$ is a
Cohen-Macaulay local ring of dimension $d$. We shall prove that if $M$ is a Cohen-Macaulay
$R$- module of Krull dimension $n$ such that $H_{\frak m}^n(M)$ is
of finite Gorenstein injective dimension, then ${\rm Gid}_R(H_{\frak
m}^n(M))=d-n$ (see Theorem 3.10). In the end of this paper, we prove that
if $M$ is a Matlis reflexive strongly
torsion free $R$-module with ${\rm Gfd}_RM<\infty$, then ${\rm Gfd}_{\hat{R}}\hat{M}<\infty$ (see Theorem 3.11). As a conclusion of this theorem, we prove that if ${\rm Gfd}_RM<\infty$, then ${\rm Gid}_R({\rm
Hom}_R(\hat{M},\omega))<\infty$, where $\omega$ is the dualizing module of $\hat{R}$.

\section{Preliminaries}

\hspace{0.4cm}In this section we recall some definitions that we use later.\\

\begin{definition}
Xu [Definitions 5.4.2 and 5.4.1] has introduced the notion of a
strongly torsion free and of a strongly cotorsion module. An
$R$-module $M$ is said to be {\it strongly torsion free (strongly
cotorsion)} if ${\rm Tor}_1^R(F,M)=0$ (${\rm Ext}_R^1(F,M)=0$) for
any $R$-module $F$ of finite flat dimension. One can easily show
that $M$ is strongly torsion free (strongly cotorsion) if ${\rm
Tor}_i^R(F,M)=0$ (${\rm Ext}_R^i(F,M)=0$) for any $i\geq 1$ and any
$R$-module $F$ of finite flat dimension.
\end{definition}

\medskip

\medskip

\begin{definition}
Following [EJ1], an $R$-module $N$ is said to be {\it Gorenstein
injective (or G-injective)} if there exists a
${\rm{Hom}}(\mathcal{I},-)$ exact exact sequence
$$\dots\To E_1\To E_0\To E^0\To E^1\dots$$ of injective
$R$-modules such that $N={\rm{Ker}}(E^0\To E^1)$. Dually, an
$R$-module $M$ is said to be {\it Gorenstein flat (or G-flat)} if
there exists an $\mathcal{I}\bigotimes-$ exact exact sequence
$$\dots\To F_1\To F_0\To F^0\To F^1\To \dots$$ of flat modules
such that $M={\rm{Ker}}(F^0\To F^1)$.
\end{definition}

\medskip
\begin{definition}
We can define the {\it G-injective dimension}
of an $R$-module $N$, Gid$_RN$ as follows\\
${\rm{Gid}}_RN=\inf\{n\in\mathbb{N}_0|$there exists a $\mathcal{GI}
-$resolution $N\To\mathbf{E}$ of length$\leq n\}$, where
$\mathcal{GI}$ denotes the class of all G-injective modules. We also
denote by $\widetilde{\mathcal{I}}$ ($\widetilde{\mathcal{GI}})$ the
class of all modules of finite injective dimension (finite
G-injective dimension). Dually, we can define the {\it G-flat
dimension}
of an $R$-module $M$, Gfd$_RM$ as follows\\
${\rm{Gfd}}_RM=\inf\{n\in\mathbb{N}_0|$there exists a $\mathcal{GF}
-$resolution $\mathbf{F}\To M$ of length$\leq n\}$, where
$\mathcal{GF}$ denotes the class of all G-flat modules and a
$\mathcal{GF} -$resolution $\mathbf{F}\To M$ of length$\leq n$ is an
exact sequence $0\To F_m\To F_{m-1}\To\dots\To F_0\To M\To 0$ of
$R$-modules such that each $F_i$ is G-flat and $m\leq n$. We also
denote by $\widetilde{\mathcal{F}}$ ($\widetilde{\mathcal{GF}})$ the
class of all modules of finite flat dimension (finite G-flat
dimension).\\
\end{definition}

\medskip

\begin{definition}
 Over a local ring $(R,\frak m)$ an $R$-module $M$ is said to be {\it
Matlis reflexive} if
$M\cong\mbox{Hom}_{R}(\mbox{Hom}_{R}(M,E(R/\frak m)),E(R/\frak m))$.

\hspace{0.4cm}It should be noted that if $M$ is a Matlis reflexive
$R$-module, then any submodule and any quotient of $M$ is Matlis
reflexive. Enochs in [E] proved that if the ring $R$ is local (not
necessarily complete) and $M$ is a Matlis reflexive, then it
contains a finitely generated submodule $S$ such that $M/S$ is
Artinian. In view of these arguments, one can easily show that if
$M$ is a Matlis reflexive $R$-module, then there is an isomorphism
$M\cong M\otimes_R\hat{R}$.
\end{definition}

\medskip
\begin{definition}
Let $(R,\frak m)$ be a local ring and let $M$ be an $R$-module.
Following [F], the
{\it width} of $M$ is \\
\begin{center}
width$_RM={\rm inf}\{m\in\mathbb{Z}|{\rm Tor}_m^R(k,M)\neq 0\}.$
\end{center}
In view of the Nakayama's lemma, on can easily see that if $M$ is a
non-zero finitely generated $R$-module, then width$_RM=0$.
\end{definition}

\medskip
\begin{definition}
Let $X$ be a complex and let $m$ be an integer. $\Sigma^m X$ denotes the complex $X$ shifted m degrees (to the left); it is given by
$(\Sigma^m X)_l=X_{l-m}$ and $\partial_l^{\Sigma^m X}=(-1)^m\partial_{l-m}^X$ for each $l\in\mathbb{Z}$. Also the supremum $\sup X$ and the infimum $\inf X$ of the complex $X$ are defined by, respectively, $$\sup X=\sup\{l\in\mathbb{Z}| {\rm H}_l(X)\neq 0\} \hspace{0.25cm} {\rm and}$$
$$\inf X=\inf\{l\in\mathbb{Z}| {\rm H}_l(X)\neq 0\}$$
where H$_l(X)$, $l$-th homology module of the complex $X$, is H$_l(X)={\rm Ker}\partial_l^X/{\rm Im}\partial_{l+1}^X$ for each $l\in\mathbb{Z}$. As the convention $\sup\emptyset=-\infty$ and $\inf\emptyset=\infty$. When H$_l(X)=0$ for all $l\in\mathbb{Z}$, we set $\sup X=-\infty$ and $\inf X=\infty$. We note that $\sup\Sigma^m X= \sup X+m$ and $\inf\Sigma^m X=\inf X+m$ for each $m\in\mathbb{Z}$.
\end{definition}

\medskip
\begin{definition}
Let $X$ and $Y$ be two complexes of $R$-modules such that $H(Y)$ is
right bounded. Then there exists a quasi-isomorphism (a homology
isomorphism) $P\rightarrow Y$ with $P$ a right bounded complex of
projective $R$-modules (the complex $P$ is usually called a {\it
projective resolution of $Y$}). The symbol ${\bf R}{\rm Hom}(Y,X)$
is defined to be the equivalent class (under quasi isomorphism) of
$R$-complexes represented by ${\rm Hom}_R(P,X)$, for any projective
resolution $P\rightarrow Y$. If $H(X)$ is left bounded, then there
exists a quasi isomorphism $X\rightarrow I$ with $I$ a bounded above
complex of injective $R$-modules( the complex $I$ is usually called
an {\it injective resolution of $X$}). So we can also define ${\bf
R}{\rm Hom}(Y,X)$ as the equivalent class (under quasi isomorphism)
of $R$-complexes represented by ${\rm Hom}_R(Y,I)$, for any
injective resolution $X\rightarrow I$. In particular if $X=N$ and
$Y=M$ be two $R$-modules, then we can apply each one of these
approaches and in this case, for each $i\geq 0$, we have ${\rm
Ext}_R^i(M,N)=H_{-i}({\bf R}{\rm Hom}(M,N))$
\end{definition}

\medskip
\begin{definition}
Let $R$ be a commutative Noetherian ring and let $D(R)$ denote the
derived category of $R$-complexes (see Hartshorn's book [Ha]). The
complex $D\in D(R)$ is a {\it dualizing complex} for $R$ if $D$ has
the following
conditions:\\
\begin{list}{}{\setlength{\rightmargin}{\leftmargin}}
\item[(i)] $D$ has finite homology modules.
\item[(ii)] $D$ has finite injective dimension.
\item[(iii)] The canonical morphism $\theta: R\To {\bf R}{\rm
Hom}_R(D,D)$ is an isomorphism.
\end{list}
A dualizing complex $D$ is said to be {\it normalized dualizing
complex} if sup$D={\rm dim}R$, where sup$D={\rm sup}\{n\in
\mathbb{Z}|H_n(D)\neq 0\}$
\end{definition}

\medskip
\begin{definition}
An $R$-module $\omega$ is said to be a dualizing module for $R$ if
$\omega$ is a dualizing complex (concentrated in degree zero) for
$R$.
\end{definition}
\medskip
Similar to the G-injective dimension of an $R$-module we can define
the G-injective dimension a complex as follows.

\medskip
\begin{definition}
Let $C_{\Box}(R)$ denote the category of all bounded $R$-complexes,
let $C_{\sqsubset}(R)$ denote the category of $R$-complexes bounded
at the left and let $D_{\sqsubset}(R)$ denote homologically
left-bounded complexes (see [F]). A complex $Y\in D_{\sqsubset}(R)$
is said to be of  {\it finite G-injective dimension} if there exists
a complex $B$ in $C_{\Box}(R)$ such that $Y\simeq B$ where the sign
$\simeq$ denotes the quasi isomorphism between two complexes. The
{\it G-injective dimension}, Gid$_RY$, of $Y\in D_{\sqsubset}(R)$ is
defined as
\begin{center}
Gid$_RY=\inf\{\sup\{l\in\mathbb{Z}|B_{-l}\neq 0\}| B\in
C_{\sqsubset}(R)$ is isomorphic to $Y$ in D(R)\\ and every $B_l$ is
G-injective$\}$
\end{center}
\end{definition}

\medskip
\begin{definition}
 Let $R$ be a Cohen-Macaulay local ring of Krull dimension
$d$ admitting a dualizing module $\omega$ and with residue field
$k$. Following [EJX], $\mathcal{G}_0(R)$ denote the class of
$R$-modules $M$ such that ${\rm Tor}_i^R(\omega,M)={\rm
Ext}_R^i(\omega,\omega\otimes_RM)=0$ for all $i>0$ and such that the
natural map $M\To {\rm Hom}(\omega,\omega\otimes_RM)$ is an
isomorphism. This class of $R$-modules is called {\it Auslander
class}. Also, $\mathcal{J}_0(R)$ denote the class of $R$-modules $N$
such that Ext$_R^i(\omega,N)={\rm Tor}_i^R(\omega,{\rm
Hom}_R(\omega,N))=0$ for all $i>0$ and such that the natural map
$\omega\otimes_R{\rm Hom}_R(\omega,N)\To N$ is an isomorphism. This
class of $R$-modules is called {\it Bass class}. It should be noted
that $\mathcal{G}_0(R)$ and $\mathcal{J}_0(R)$ are also called {\it
Foxby classes}.
\end{definition}

\medskip
\begin{remark}
 It follows from [EJ1, Proposition 10.4.23]
that $\mathcal{J}_0(R)=\widetilde{\mathcal{GI}}$. Moreover,
according to [EJ1, Proposition 10.4.17 and Corollary 10.4.29] we
have $\mathcal{G}_0(R)\subseteq\widetilde{\mathcal{GF}}$ and
according to [EJ1, Theorem 10.4.10 and Theorem 10.4.28] we have
$\widetilde{\mathcal{GF}}\subseteq \mathcal{G}_0(R)$. Therefore we
can deduce that  $\mathcal{G}_0(R)=\widetilde{\mathcal{GF}}$. When
$R$ is Gorenstein, each of the classes $\mathcal{G}_0(R)$ and
$\mathcal{J}_0(R)$ contains all $R$-modules.
\end{remark}

\section{The main results}

\hspace{0.4cm}It follows from the proof of [S1, Theorem 3.1] that
 if $\frak a$ is an ideal of a
commutative Noetherian ring $R$ (without any other conditions on
$R$) and $G$ is a G-injective $R$-module, then $H_{\frak a}^i(G)=0$
for all $i>0$. This fact leads us to the following proposition.

\medskip

\begin{proposition}
Let $R$ admit a dualizing complex $D$ and let $\frak a$ be an ideal
of $R$. If $G$ is a G-injective $R$-module, then ${\rm
Gid}_R(\Gamma_{\frak a}(G))<\infty$.
\end{proposition}
\begin{proof}
It follows from [CFH, Theorem 5.9] that Gid$_R({\bf R}\Gamma_{\frak
a}(G))<\infty$. In view of [S1, Theorem 3.1], for each $i>0$ we have
$H_{\frak a}^i(G)=0$. Then for each $i>0$, we have $H^i({\bf
R}\Gamma_{\frak a}(G))=H_{\frak a}^i(G)=0$. On the other hand, since
Gid$_R({\bf R}\Gamma_{\frak a}(G))<\infty$, there exists a
non-negative integer $t$ and the following complex of $R$-modules
$${\bf G}=0\To G^0\To G^1\To\dots\To G^t\To 0$$ such that each
$G^i$ is G-injective and we have the quasi isomorphism of complexes
${\bf R}\Gamma_{\frak a}(G)\simeq{\bf G}$. Therefore $H^0({\bf
G})\cong \Gamma_{\frak a}(G)$ and $H^i({\bf G})=0$ for all $i>0$.
This fact implies that the complex
$$0\To \Gamma_{\frak a}(G)\To G^0\To G^1\To\dots\To G^t\To 0$$ is
acyclic and so it is a finite G-injective resolution of
$\Gamma_{\frak a}(G)$. Therefore Gid$_R(\Gamma_{\frak
a}(G))<\infty$.
\end{proof}

\medskip

\begin{theorem}
Let $R$ be a commutative Noetherian ring of Krull dimension $d$
admitting a dualizing complex $D$ and let $\frak a$ be an ideal of
$R$. If $G$ is a G-injective $R$-module, then $\Gamma_{\frak a}(G)$
is G-injective.
\end{theorem}
\begin{proof}
As $G$ is G-injective, there is an exact sequence of $R$-modules
$$\dots\To E_1\To E_0\To G\To 0$$ such that each $E_i$ is
injective and $K_i={\rm Ker}(E_i\To E_{i-1})$ is G-injective.
Application of the functor $\Gamma_{\frak a}(-)$ to the above exact
sequence and using [S1, Theorem 3.1], we get the following exact
sequence
$$\dots\To \Gamma_{\frak a}(E_1)\To \Gamma_{\frak a}(E_0)\To
\Gamma_{\frak a}(G)\To 0.$$ According to Proposition 3.1, we
have Gid$_R(\Gamma_{\frak a}(K_i))<\infty$ for each $i$. On the
other hand, if $N$ is any $R$-module of finite G-injective
dimension, in view of [CFH, Theorem 6.8] there exist a prime ideal
$\frak p$ of $R$ such that Gid$_RN={\rm depth}R_{\frak p}-{\rm
width}_{R_{\frak p}}N_{\frak p}$. This fact implies that the
following inequalities
$${\rm Gid}_RN\leq {\rm depth}R_{\frak p}\leq {\rm dim}R_{\frak
p}={\rm hight}\frak p\leq d.$$ Thus Gid$_R(\Gamma_{\frak
a}(K_i))\leq d$ for each $i$. Now, let $E$ be any injective
$R$-module. In view of [H, Theorem 2.22] for each $i>d$, there are
the following isomorphisms $$0={\rm Ext}_R^i(E,\Gamma_{\frak
a}(K_{d-1}))\cong {\rm Ext}_R^{i-1}(E,\Gamma_{\frak
a}(K_{d-2}))\cong\dots\cong {\rm Ext}_R^{i-d}(E,\Gamma_{\frak
a}(G)).$$ So for each $i>0$, we have ${\rm
Ext}_R^{i}(E,\Gamma_{\frak a}(G))=0$. Now, it follows from [H,
Theorem 2.22] that $\Gamma_{\frak a}(G)$ is G-injective.
\end{proof}
\medskip

\begin{corollary}
Let $R$ be a commutative Noetherian ring of Krull dimension $d$
admitting a dualizing complex $D$, let $\frak a$ be any ideal of $R$ and let $X\in D_{\sqsubset}(R)$. Then ${\rm
Gid}_R{\bf R}\Gamma_{\frak a}(X)\leq {\rm Gid}_R X$.
\end{corollary}
\begin{proof}
If Gid$_RX=\infty$, the inequality is clear. Let Gid$_RX=t$ and let
$$\mathcal{I}:=0\To I_0\To I_{-1}\To\dots\To I_{-t}\To I_{-(t+1)}\To \dots$$ be an injective resolution for
$X$. Then $K_{-i}={\rm Ker}(I_{-i}\To I_{-(i+1)})$, $i$-th cosyzygy of $\mathcal{I}$, is G-injective for each $i\geq t$. We note that
${\bf R}\Gamma_{\frak a}(X)\simeq\Gamma_{\frak a}(\mathcal{I})$ and
hence $\Gamma_{\frak a}(\mathcal{I})$ is an injective resolution of
${\bf R}\Gamma_{\frak a}(X)$. On the other hand, it follows from
Theorem 3.2 that $\Gamma_{\frak a}(K_{-i})$ is G-injective for
all $i\geq t$ and so this implies that Gid$_R{\bf R}\Gamma_{\frak
a}(X)\leq t$.
\end{proof}

\medskip

\begin{theorem}
Let $R$ be a commutative Noetherian ring, let $\frak a$ be any ideal of $R$ and let $X\in D_{\sqsubset}(R)$. Then
$-{\rm inf}{\bf R}\Gamma_{\frak a}(X)\leq {\rm Gid}_R X$.
\end{theorem}
\begin{proof}
If Gid$_RX=\infty$, the inequality is clear. Let Gid$_RX=t$
and let $$\mathcal{I}:=0\To I_0\To I_{-1}\To\dots\To I_{-t}\To I_{-(t+1)}\To \dots$$ be an injective resolution for
$X$. Then $K_{-i}={\rm Ker}(I_{-i}\To I_{-(i+1)})$, $i$-th cosyzygy
of $\mathcal{I}$, is G-injective for each $i\geq t$. As $-\inf X\leq
t$, the complex $$0\To I_{-t}\To I_{-(t+1)}\To\dots$$ is an
injective resolution for $K_{-t}$. On the other hand, it follows
from [S1, Theorem 3.1] that $$0\To \Gamma_{\frak a}(I_{-t})\To
\Gamma_{\frak a}(I_{-(t+1)})\To\dots$$ is an injective resolution
for $\Gamma_{\frak a}(K_{-t})$. Therefore $H_i({\bf R}\Gamma_{\frak
a}(X))=0$ for all $i\leq -t$ which this implies the assertion.
\end{proof}

\medskip
\hspace{0.5cm} The authors in [ET2, Theorem 2.5] proved that over a
commutative Noetherian local ring, an $R$-module $M$ is G-injective
if and only if it is cotorsion and ${\rm Hom}_R(\hat{R},M)$ is
G-injective as an $\hat{R}$-module. This fact leads us to the
following two lemmas:

\medskip
\begin{lemma}
Let $(R,\frak m)$ be a local ring of Krull dimension $d$ and $M$ be
an Artinian $R$-module. Then $M$ is a G-injective $R$-module if and
only if it is G-injective as $\hat{R}$-module.
\end{lemma}
\begin{proof}
As $M$ is Artinian, there exits an exact sequence of $R$-modules
$$0\To M\To E(R/\frak m)^{t_0}\To E(R/\frak m)^{t_1}$$ such that
$t_0$ and $t_1$ are positive integer. Applying the functor ${\rm
Hom}_R(\hat{R},-)$ to this exact sequence and using the isomorphism
Hom$_R(\hat{R}, E(R/\frak m))\cong E(R/\frak m)$ induce the
following exact sequence of $R$-modules
$$0\To{\rm
Hom}_R(\hat{R},M)\To E(R/\frak m)^{t_0}\To E(R/\frak m)^{t_1}.$$ The
five lemma implies the isomorphism ${\rm Hom}_R(\hat{R},M)\cong M$.
Now we prove the assertion. At first assume that $M$ is a
G-injective $R$-module. Using the isomorphism ${\rm
Hom}_R(\hat{R},M)\cong M$ and [ET2, Theorem 2.5] we conclude that
$M$ is a G-injective $\hat{R}$-module. Conversely, assume that $M$
is a G-injective $\hat{R}$ module. So there exists an exact sequence
of injective $\hat{R}$-modules $$\dots\To E_1\To E_0\To M\To 0$$
such that each $K_i={\rm Ker}(E_i\To E_{i-1})$ is a G-injective
$\hat{R}$-module. Let $F$ be an $R$-module of finite flat dimension.
Since each injective $\hat{R}$-module is an injective $R$-module,
there is the following isomorphisms $${\rm Ext}_R^1(F,M)\cong {\rm
Ext}_R^2(F,K_0)\cong{\rm Ext}_R^3(F,K_1)\cong \dots\cong{\rm
Ext}_R^{d+1}(F,K_{d-1}).$$ It follows from [RG, p. 84] that
pd$_RF\leq d$. This fact implies that $M$ is a strongly cotorsion
$R$-module and so is a cotorsion $R$-module. Now the result follows
by [ET2, Theorem 2.5].
\end{proof}

\medskip

\medskip
\begin{lemma}
Let $(R,\frak m)$ be a local ring and $M$ be an Artinian $R$-module.
Then ${\rm Gid}_RM={\rm Gid}_{\hat{R}}M$.
\end{lemma}
\begin{proof}
Let Gid$_RM=s$. Since $M$ is Artinian, using [EJ2, Lemma 4.2], there
exists an exact sequence of $R$-modules $$0\To M\To E(R/\frak
m)^{t_0}\To E(R/\frak m)^{t_1}\To \dots$$ such that each $t_i$ is
positive integer and $K^s={\rm Ker}(E(R/\frak m)^{t_s}\To E(R/\frak
m)^{t_{s+1}})$ is an Artinian G-injective $R$-module. It follows
from Lemma 3.5 that $K^s$ is G-injective as an
$\hat{R}$-module and so Gid$_{\hat{R}}M\leq s$. Conversely assume
that Gid$_{\hat{R}}M=r$. By a similar proof and using Lemma 3.5, we deduce that Gid$_RM\leq r$, and hence the assertion
follows.
\end{proof}

\medskip

\begin{theorem}
Let $(R,\frak m)$ be a local ring, and let $M$ be a finitely
generated $R$-module. Then ${\rm Gid}_R{\bf R}\Gamma_{\frak
m}(M)<\infty$ if and only if ${\rm Gid}_{\hat{R}}(M\otimes_R\hat{R})<\infty$.
\end{theorem}
\begin{proof}
Let ${\rm Gid}{\bf R}\Gamma_{\frak m}(M)=t$ and let
$$\mathcal{I}:=0\To I_0\To I_{-1}\To\dots\To I_{-t}\To I_{-(t+1)}\To
\dots$$ be an injective resolution of $M$. Then there is the
following quasi isomorphisms
$${\bf R}\Gamma_{\frak
m\hat{R}}(M\otimes_R\hat{R})\simeq\Gamma_{\frak
m\hat{R}}(\mathcal{I}\otimes_R\hat{R})\simeq\Gamma_{\frak
m}(\mathcal{I}).$$ Let $K_{-i}={\rm Ker}(I_{-i}\To I_{-(i+1)})$,
$i$-th cosyzygy of $\mathcal{I}$. Since ${\rm Gid}_R{\bf R}\Gamma_{\frak m}(M)=t$, we conclude that
 $\Gamma_{\frak m}(K_{-i})$ is a G-injective artinian $R$-module for
each $i\geq t$. It follows from Lemma 3.6 that $\Gamma_{\frak
m}(K_{-i})$ is G-injective $\hat{R}$-module for each $i\geq t$ and
hence Gid$_{\hat{R}}{\bf R}\Gamma_{\frak m\hat{R}}(M\otimes_R\hat{R})\leq t$.
Now, since $\hat{R}$ admits a dualizing complex, the result follows
from [CFH, Theorem 5.9]. Conversely, assume that ${\rm Gid}(M\otimes_R\hat{R})=s$ and let
$$\mathcal{J}:=0\To J_0\To J_{-1}\To\dots\To J_{-s}\To J_{-(s+1)}\To
\dots$$ be an $\hat{R}$- injective resolution of $M\otimes_R\hat{R}$. It follows from [EJ2, Lemma 4.2] that $L_{-i}={\rm Ker}(J_{-i}\To J_{-(i+1)})$,
 $i$-th cosyzygy of $\mathcal{J}$, is G-injective for each $j\geq s$. On the other hand, since $M\otimes_R\hat{R}$ is finitely generated, $\Gamma_{\frak m\hat{R}}(L_{-i})$ is artinian for each $i$. Thus, since $\hat{R}$ admits a dualizing complex, in view of Theorem 3.2, the module $\Gamma_{\frak m\hat{R}}(L_{-s})$ is G-injective artinian $\hat{R}$-module and so is $G$-injective $R$-module by Lemma 3.6. Therefore, Gid$_R(\Gamma_{\frak m\hat{R}}(\mathcal{J}))<\infty$. Now the result follows by the following quasi isomorphisms
$${\bf R}\Gamma_{\frak m}(M)\simeq{\bf R}\Gamma_{\frak m\hat{R}}(M\otimes_R\hat{R})\simeq \Gamma_{\frak m\hat{R}}(\mathcal{J}).$$
\end{proof}

\medskip

\begin{corollary}
Let $(R,\frak m)$ be a local ring and let $M$ be a non-zero cyclic
$R$-module such that ${\rm Gid}_R{\bf R}\Gamma_{\frak m}(M)<\infty$.
Then $R$ is Gorenstein.
\end{corollary}
\begin{proof}
It is straightforward to show that if $M$ is a cyclic $R$-module, then
$M\otimes_R\hat{R}$ is a cyclic $\hat{R}$-module too. Now, it follows
from the theorem above that
Gid$_{\hat{R}}(M\otimes_R\hat{R})<\infty$. Thus, according to [FFr,
Theorem 4.5] the ring $\hat{R}$ is Gorenstein. Therefore, $R$ is
Gorenstein.
\end{proof}

\medskip

\hspace{0.4cm} It follows from [CFH, Theorem 5.9] that if $R$ is a
commutative Noetherian ring admitting a dualizing complex, $\frak a$
is any ideal contained in radical of $R$ and $M$ is a finitely
generated $R$-module, then Gfd$_R{\bf R}\Gamma_{\frak a}(M)<\infty$
implies that Gfd$_RM<\infty$. In the following theorem, with a
different proof we prove that if $(R,\frak m)$ is a local ring and
$\frak a=\frak m$, then the result is true without the condition
existence a dualizing complex for $R$

\medskip
\begin{theorem}
Let $(R,\frak m)$ be a local ring of Krull dimension $d$ and let $M$
be a finitely generated $R$-module. If ${\rm Gfd}_R{\bf
R}\Gamma_{\frak m}(M)<\infty$, then ${\rm Gfd}_RM<\infty$.
\end{theorem}
\begin{proof}
As $M$ is finitely generated, there is a quasi isomorphism of
$R$-complexes ${\bf R}\Gamma_{\frak m}(M)\simeq {\bf R}\Gamma_{\frak
m\hat{R}}(M\otimes_R\hat{R})$. This fact and the assumption imply
that ${\rm Gfd}_{\hat {R}}{\bf R}\Gamma_{\frak
m\hat{R}}(M\otimes_R\hat{R})<\infty$. On the other hand in view of
[ET1, Corollary 3.5], if ${\rm
Gfd}_{\hat{R}}(M\otimes_R{\hat{R}})<\infty$, then ${\rm
Gfd}_RM<\infty$. So we may assume that $R$ is a complete local ring;
and hence $R$ admits a dualizing complex $C$. If we consider $n=d-{\rm
sup}C$, then $D=\Sigma^nC$ is a normalized dualizing complex for $R$, where $\Sigma^nC$
denotes the complex $C$ shifted (to the left) $n$ degrees
(see [F, 15.7]). According to the local duality theorem for local
cohomology there is a quasi isomorphism of $R$-complexes ${\bf
R}\Gamma_{\frak m}(M)\simeq \Sigma^{-d}{\rm Hom}_R({\rm
Hom}_R(M,D),E(R/\frak m))$. The assumption ${\rm Gfd}_R{\bf
R}\Gamma_{\frak m}(M)<\infty$ and [CFH, Theorem 5.3] imply that
${\rm Gid}_R(({\rm Hom}_R(M,D))<\infty$; and hence there is a
complex
$$\mathcal{G}=0\To G_0\To G_{-1}\To \dots\To G_{-t}\To 0$$ such that $t\in
\mathbb{N}_0$ and each $G_i$ is G-injective and there is the quasi
isomorphism ${\rm Hom}_R(M,D)\simeq\mathcal{G}$. It follows from [F,
15.10] that $M={\rm Hom}_R({\rm Hom}_R(M,D),D)$ and so in order to
prove our claim, it suffices to show that ${\rm Gfd}_R({\rm
Hom}_R({\rm Hom}_R(M,D),D))<\infty$. In view of [F, 6.15], there is
a quasi isomorphism of $R$-complexes ${\rm Hom}_R({\rm
Hom}_R(M,D),D)\simeq {\rm Hom}_R(\mathcal{G},D)$. Let $$D=0\To
I_0\To I_{-1}\To \dots\To I_{-n}\To 0$$ where $n\in\mathbb{N}_0$ and
each $I_i$ is injective. We note that for each $j\geq 0$, there is
${\rm Hom}_R(\mathcal{G},D)_j=\bigoplus_{i=0}^{t}{\rm
Hom}_R(G_{-i},I_{-i+j})$. It follows from [CFH, Theorem 5.3] that
each ${\rm Hom}_R(\mathcal{G},D)_j$ is G-flat. On the other hand
the quasi isomorphism $M={\rm Hom}_R({\rm Hom}_R(M,D),D)\simeq {\rm
Hom}_R(\mathcal{G},D)$ of $R$-complexes implies the following exact
sequence $R$-modules $$0\To{\rm Hom}_R(\mathcal{G},D)_s\To{\rm
Hom}_R(\mathcal{G},D)_{s-1}\To \dots\To{\rm Hom}_R(\mathcal{G},D)_
0\To M\To 0$$ where $s\leq t+n$. Thus Gfd$_RM<\infty$.
\end{proof}

\medskip

\begin{theorem}
Let $(R,\frak m)$ be a Cohen-Macaulay local ring of dimension $d$ and let $M$ be a
Cohen-Macaulay $R$-module of dimension $n$. If $H_{\frak m}^n(M)$ is
of finite G-injective dimension, then ${\rm Gid}_RH_{\frak
m}^n(M)=d-n.$ In particular, if $M$ is maximal Cohen-Macaulay, then
$H_{\frak m}^d(M)$ is G-injective.
\end{theorem}
\begin{proof}
As $H_{\frak m}^n(M)$ is Artinian, there are the following
isomorphisms $$H_{\frak m}^n(M)\cong H_{\frak
m}^n(M)\otimes_R\hat{R}\cong H_{\frak
m\hat{R}}^n(M\otimes_R\hat{R}).$$ We note that $M\otimes_R\hat{R}$
is a Cohen-Macaulay $\hat{R}$-module of dimension $n$ and by Lemma 3.6, we have Gid$_R(H_{\frak m}^n(M)={\rm
Gid}_{\hat{R}}(H_{\frak m}^n(M)\otimes_R\hat{R})= {\rm
Gid}_{\hat{R}}(H_{\frak m\hat{R}}^n(M\otimes_R\hat{R}))$. So we may
assume that $R$ is a complete local ring and hence $R$ has a dualizing module. Using the local duality
theorem for local cohomology, we have the isomorphism $H_{\frak
m}^n(M)\cong {\rm Hom}_R({\rm Ext}_R^{d-n}(M,\omega), E(R/\frak m))$
where $\omega$ is a dualizing module of $R$. It follows from
[BH, Theorem 3.3.10] that ${\rm Ext}_R^{d-n}(M,\omega)$ is a
Cohen-Macaulay $R$-module of dimension $n$ and by the previous
isomorphism and the assumption it has finite G-flat dimension. Now,
in view of [H1, Theorem 3.19] and [CFFr, Theorem 2.8], we have
Gfd$_R({\rm Ext}_R^{d-n}(M,\omega))={\rm depth}R-{\rm depth}({\rm
Ext}_R^{d-n}(M,\omega))=d-n.$ The fact that $E(R/\frak m)$ is
injective cogenerator and [CFH, Theorem 5.3] imply that
Gid$_R(H_{\frak m}^n(M))=d-n$. The second claim follows easily
because dim$M=d$.
\end{proof}

\medskip
\begin{theorem}
Let $(R,\frak m)$ be a Cohen-Macaulay local ring of dimension
$d$ and let $M$ be a Matlis reflexive strongly
torsion free $R$-module. If ${\rm Gfd}_RM<\infty$, then ${\rm Gfd}_{\hat{R}}\hat{M}<\infty$.
\end{theorem}
\begin{proof}
There are the following isomorphisms
$$\hat{M}\cong\underset{\underset{n}{\leftarrow}}{\mbox{lim}}M/{\frak
m}^nM\cong\underset{\underset{n}{\leftarrow}}{\mbox{lim}} (R/{\frak
m}^n\otimes_RM)\cong
\underset{\underset{n}{\leftarrow}}{\mbox{lim}}(R/{\frak
m}^n\otimes_R{\rm Hom}_R(\check{M}, E(R/\frak
m)))$$$$\cong\underset{\underset{n}{\leftarrow}}{\mbox{lim}}({\rm
Hom}_R({\rm Hom}_R(R/{\frak m}^n,\check{M}), E(R/\frak m)))$$$$\cong
{\rm Hom}_R(\underset{\underset{n}{\rightarrow}}{\mbox{lim}}{\rm
Hom}_R(R/{\frak m}^n,\check{M}), E(R/\frak m))\cong (\Gamma_{\frak
m}(\check{M})\check{)}$$ where $(-\check{)}={\rm Hom}_R(-,E(R/\frak
m))$ is the Matlis duality functor. One can easily show that $\check{M}$
is strongly cotorsion and Matlis reflexive. Application of [S2,
Corollary 3.3] implies that $\Gamma_{\frak m}(\check{M})$ is
strongly cotorsion and Matlis reflexive. On the other hand, since
$M$ is of finite G-flat dimension, one easily conclude that
$\check{M}$ is of finite G-injective dimension. The fact that
$\check{M}$ is a Matlis reflexive $R$-module implies the isomorphism
$\check{M}\cong \check{M}\otimes_R\hat{R}$ and since $\check{M}$ is
strongly cotorsion, we have ${\rm Ext}_R^i(\hat{R},\check{M})=0$ for
all $i>0$. Moreover, there are the following isomorphisms $${\rm
Hom}_R(\hat{R},\check{M})\cong{\rm Hom}_R(M\otimes_R\hat{R},
E(R/\frak m))\cong {\rm Hom}_R(M, E(R/\frak m))\cong \check{M}.$$
Thus it follows from [ET2, Theorem 2.7] that
$\check{M}\in\mathcal{J}_0(\hat{R})$. Therefore the equality
$\Gamma_{\frak m}(\check{M})=\Gamma_{\frak m\hat{R}}(\check{M})$ and
[S3, Proposition 4.9] imply that $\Gamma_{\frak m}(\check{M})\in
\mathcal{J}_0(\hat{R})$ and so $\Gamma_{\frak m}(\check{M})$ is of
finite G-injective dimension as $\hat{R}$-module by [EJ1,
Proposition 10.4.23]. Now, in view of [CFH, Theorem 5.3], we conclude that $(\Gamma_{\frak m}(\check{M})\check{)}$ is of finite
G-flat dimension as $\hat{R}$-module which this fact gives the assertion.
\end{proof}

\medskip

\begin{corollary}
Let $(R,\frak m)$ be a Cohen-Macaulay local ring of dimension
$d$ and let $M$ be a Matlis reflexive strongly
torsion free $R$-module. If ${\rm Gfd}_RM<\infty$, then ${\rm Gid}_R({\rm
Hom}_R(\hat{M},\omega))<\infty$, where $\omega$ is the dualizing module of $\hat{R}$.
\end{corollary}
\begin{proof}
 It follows from [S2, Theorem 3.2] that $\hat{M}$ is a maximal
Cohen-Macaulay $\hat{R}$-module and ${\rm Hom}_{\hat{R}}(\hat{M},\omega)\in
\mathcal{J}_0(\hat{R})$ by [S3, Theorem
4.14]. Since $\hat{M}$ is a Matlis
reflexive $\hat{R}$-module, we have an isomorphism
$\hat{M}\otimes_R\hat{R}\cong\hat{M}$ which implies the isomorphism ${\rm
Hom}_{\hat{R}}(\hat{M},\omega)\cong {\rm Hom}_{R}(\hat{M},\omega)$.
On the other hand, since $\hat{M}$ is Matlis reflexive, there are the following isomorphisms
$${\rm Hom}_R(\hat{R},{\rm Hom}_{R}(\hat{M},\omega))\cong {\rm Hom}_R(\hat{R}\otimes_R\hat{M},\omega)\cong {\rm Hom}_{R}(\hat{M},\omega).$$
Moreover by considering this fact that any $\hat{R}$-injective resolution of ${\rm Hom}_{R}(\hat{M},\omega)$ is an $R$-injective resolution of ${\rm Hom}_{R}(\hat{M},\omega)$ and for any injective $\hat{R}$-module $E$ there is an isomorphism ${\rm Hom}_R(\hat{R},E)\cong E$, we deduce that ${\rm Ext}_R^i(\hat{R}, {\rm Hom}_{R}(\hat{M},\omega))=0$ for all $i>0$. Now, the assertion follows by [ET2, Theorem 2.7].
\end{proof}

\end{document}